\documentclass[11pt]{amsart}

\usepackage{amssymb,tikz}
\usepackage{amsthm,dsfont}
\usepackage{amsmath, a4wide}
\usepackage{amsbsy}
\usepackage[all]{xy}
\usepackage{bm}
\usepackage[hidelinks]{hyperref}
\usepackage{color}
\usepackage[capitalize]{cleveref}

\usepackage{subfigure}
\usepackage{mathtools}

 \usepackage{paralist}
 
 \usepackage{todo}

\usepackage[font=small,labelfont=bf]{caption}

\newtheorem{thm}{Theorem}[section]
\newtheorem{cor}[thm]{Corollary}

\newtheorem{lem}[thm]{Lemma}
\newtheorem{defi}[thm]{Definition}
\newtheorem{prop}[thm]{Proposition}

\theoremstyle{definition}

\newtheorem{ex}[thm]{Example}

\newcommand{\NN}{\mathbb{N}}

\newcommand{\RR}{\mathbb{R}}

\DeclareMathOperator{\Cech}{Cech}

\DeclareMathOperator{\rank}{rank}
\newcommand{\Ball}{\overline B}

\DeclareMathOperator{\diam}{diam}
\DeclareMathOperator{\bd}{bd}

\DeclareMathOperator{\comp}{comp}
\DeclareMathOperator{\conn}{conn}
\DeclareMathOperator{\sep}{sep}

\DeclareMathOperator{\Count}{Count}
\DeclareMathOperator{\SubsetCount}{RelativeCount}

\begin{document}

\title{Persistent Betti numbers of random \v Cech complexes}

\author[U. Bauer]{Ulrich Bauer} 
\address{Technical University of Munich\\ Zentrum Mathematik (M10)\\ Boltzmannstr. 3 \\ 85748 Garching \\ Germany} 

\author[F. Pausinger]{Florian Pausinger} 
\address{School of Mathematics \& Physics, Queen's University Belfast, BT7 1NN, Belfast, United Kingdom.}

\date{}

\begin{abstract}
We study the persistent homology of random \v Cech complexes. 
Generalizing a method of Penrose for studying random geometric graphs, we first describe an appropriate theoretical framework in which we can state and address our main questions. Then we define the $k$th \emph{persistent} Betti number of a random \v Cech complex and determine its asymptotic order in the subcritical regime. This extends a result of Kahle on the asymptotic order of the ordinary $k$th Betti number of such complexes to the persistent setting.

\end{abstract}

\maketitle

\section{Introduction}
\label{sec1}

\subsection{Motivation}
In this paper we study the persistent homology of random geometric complexes, which are simplicial complexes built on the points of a Poisson point process in $d$-dimensional Euclidean space.
In particular, we focus on  \v Cech complexes, which are homotopy equivalent to a union of balls of the same radius centered at the points; however, analogous results can also be obtained for Vietoris--Rips complexes using essentialy the same arguments.
Our work is motivated by recent results on the topology of random geometric complexes (see \cite{KahleSurvey} for a survey) as well as topological data analysis, an emerging field that received increasing attention over the last years. In this context, point cloud data is analyzed with methods from persistent homology. The goal of a statistical interpretation of the results requires a probabilistic null hypothesis to compare the results to.
Thus, it is of utmost importance to understand the persistent homology of randomly sampled point sets.
We provide results in this direction, drawing motivation for our investigations on results of Kahle \cite{Kahle2011Random} on the homology of geometric complexes built from random point sets. In that paper, Kahle determined the expected Betti numbers of random \v Cech and Vietoris--Rips complexes asymptotically (as the number of points, $n$, goes to $\infty$), and identified four different regimes for the radius parameter $r_n$, in dependence of $n$, with qualitatively different behavior.

Our main contribution is twofold. First, we describe an appropriate theoretical framework in which we can state and address our main questions, 
reformulating central results of Penrose \cite{Penrose2003Random} about random geometric graphs in a more general setting of \emph{geometric properties}.
In a second step, we extend the results obtained by Kahle for the \emph{subcritical regime}, i.e., for radii $r_n$ with ${r_n}^d n \rightarrow 0$ as $n \rightarrow \infty$. In particular, we determine the asymptotic order of the expected $k$th persistent Betti number of a random \v Cech complex, recovering the original result of Kahle on the ordinary $k$th Betti number of such complexes as a special case.

The persistent homology of random geometric complexes has also received attention in recent work by Bobrowski, Kahle, and Skraba \cite{Bobrowski2016Maximally}, who show the asymptotic order of the expected maximum persistence arising in a filtration of geometric complexes. Our work complements this by showing the expected rank of persistent homology with fixed persistence.

\subsection{Preliminaries}
Let $P=\{x_1, \ldots, x_N \}$ %
be a collection of points in  $\RR^d$ and let $r>0$. The \v{C}ech complex $\Cech_r(P)$ is a simplicial complex defined as
\[
  \Cech_r(P)  =  \bigg\{ Q \subseteq P
                    \mid \bigcap_{x \in Q} \Ball_r (x) \neq \emptyset \bigg\} ,
\]
where $\Ball_r(x)$ denotes the closed ball of radius $r$ centered at $x$.
By the Nerve Theorem, $\Cech_r(P)$ is homotopy-equivalent to the union of balls \[\Ball_r(P)=\bigcup_{x\in P}\Ball_r(x)\] of radius $r$ centered at the points in $P$.
Furthermore, let $\lambda: \RR^d \rightarrow [0,\infty)$ be a bounded measurable function. A \emph{Poisson process} with intensity function $\lambda$ is a point process $\mathcal{P}$ in $\RR^d$ with the property that for a Borel set $A \subseteq \RR^d$, the random variable $\mathcal{P}(A)$ is Poisson distributed with parameter $\int_A \lambda(x) \,dx$ whenever this integral is finite, and if $A_1, \ldots, A_m$ are disjoint Borel subsets of $\RR^d$, then the variables $\mathcal{P}(A_i)$, $1\leq i \leq m$, are mutually independent; see \cite{Kingman1993Poisson, Penrose2003Random}.

In the context of this paper, we are interested in 
a Poisson point process on $\RR^d$ with intensity function $x \mapsto n f(x)$, where $n \in \NN$, $x \in \RR^d$ and $f$ is a bounded probability density function.

One concrete way of constructing such a Poisson point process is given as follows.
Let $f$ be a bounded probability density function on $\RR^d$, and let $x_1, x_2, \ldots$ be independent and identically distributed $d$-dimensional random variables with common density $f$. For given $n > 0$, let $N_{n}$ be a Poisson random variable, independent of $\{x_1, x_2, \ldots \}$, and let
\[
\mathcal{P}_n:= \{x_1, x_2, \ldots, x_{N_{n}}\}.
\]
It is shown in \cite[Proposition 1.5]{Penrose2003Random} that $\mathcal{P}_{n}$ is indeed a Poisson point process on $\RR^d$ with intensity $x \mapsto n f(x)$. We call a \v{C}ech complex that is built from the random set $\mathcal{P}_n$ a \emph{random \v{C}ech complex}.

We note that in our particular setting in Section \ref{sec3} and Section \ref{sec4} working with a Poisson point process $\mathcal{P}_n$ offers some technical advantages over working with a random set $\mathcal{X}_n=\{ x_1, \ldots, x_n\}$ of fixed size, also called \emph{binomial point process}. However, the results in Section \ref{sec2} are most easily obtained for a binomial process $\mathcal{X}_n$. Therefore, we briefly recall what is known as \emph{Palm theory for Poisson processes}; see \cite[Section 1.6 and 1.7]{Penrose2003Random} for a thorough discussion of this issue.

\begin{thm}[\cite{Penrose2003Random}, Theorem 1.6] \label{thm:palm}
Let $n >0$. Suppose $p \in \NN$ and suppose $h(Y,X)$ is a bounded measurable function defined on all pairs of the form $(Y,X)$ with $X$ being a finite subset of $\RR^d$ and $Y$ a subset of $X$, satisfying $h(Y,X)=0$ unless $Y$ has $p$ elements. Then
$$ E \left( \sum_{Y \subseteq \mathcal{P}_n } h(Y,\mathcal{P}_n) \right) = \frac{n^p}{p!} \, E  \left( h(\mathcal{X}_p, \mathcal{X}_p \cup \mathcal{P}_n) \right), $$
where the sum on the left-hand side is over all subsets $Y$ of the Poisson point process $\mathcal{P}_n$, and the binomial point process set $\mathcal{X}_p$ is assumed to be independent of $\mathcal{P}_n$.
\end{thm}

\subsection{Results}
Throughout this paper, we use the Landau symbols to describe the asymptotic growth of functions. In particular, we write $f \in \mathcal{O}(g)$ if there is a constant $C>0$ and an $x_0>0$ such that $|f(x)| \leq C \cdot |g(x)|$ for all $x>x_0$. Similarly, we write $f \in \Omega(g)$ if there is a constant $c>0$ and an $x_0>0$ such that $c \cdot |g(x)| \leq |f(x)|$ for all $x>x_0$. And finally, we write $f \in \Theta(g)$ if there are constants $C>c>0$  and $x_0>0$ such that $ c \cdot |g(x)| \leq |f(x)| \leq C\cdot |g(x)|$ for all $x>x_0$.

We generalize the results of Kahle \cite{Kahle2011Random} in the \emph{subcritical regime}, i.e., for radii $r_n$ satisfying ${r_n}^d n \rightarrow 0$ as $n \rightarrow \infty$.
We fix a parameter $\vartheta \geq 1$ and define the \emph{$k$th $\vartheta$-persistent Betti number} of the inclusion map $\Cech_r(P) \hookrightarrow \Cech_{\vartheta r}(P)$ for a point set $P \subset \RR^d$ as
\[
\beta_k^{\vartheta}(P,r) := \rank H_k( \Cech_r(P) \hookrightarrow \Cech_{\vartheta r}(P) ),
\]
where $H_k$ denotes the $k$th homology with coefficients in a fixed field $\mathbb K$.
Consequently, setting $\vartheta=1$, we obtain the usual Betti numbers
\[
\beta_k(P,r) := \rank H_k( \Cech_r(P) ).
\]
We briefly discuss the related case for integer coefficients in the final section.
Our main result provides asymptotic bounds for $\beta_k^{\vartheta}(\mathcal{P}_n,r_n)$ for radii $r_n$ in the subcritical regime.
\begin{thm} \label{thm1}
Let $\mathcal{P}_n$ be as defined above and let $\vartheta \geq 1$, $d \geq 2$ and $1 \leq k \leq d-1$. Let ${r_n}^d n \rightarrow 0$ for $n \rightarrow \infty$. The expected $k$th $\vartheta$-persistent Betti number of the inclusion  $\Cech_{r_n}(\mathcal{P}_n) \hookrightarrow \Cech_{\vartheta {r_n}}(\mathcal{P}_n)$ satisfies
$$ E(\beta_k^{\vartheta}(\mathcal{P}_n,r_n)) \in \Theta \left(  n  ({r_n}^{d}n)^{(m-1)} \right )$$
as $n \rightarrow \infty$, where $m=m(\vartheta,k)$ is an integer depending on $\vartheta$ and $k$.
\end{thm}
See \cref{fig:example} for an example with parameters chosen such that the persistent Betti number behaves as $\Theta(1)$.
For $\vartheta=1$, we have $m(1,k)=k+2$, and the above theorem matches a result of Kahle \cite[Theorem 3.2]{Kahle2011Random} about the expected Betti number of \v Cech complexes:
\begin{cor}
The expected $k$th Betti number of a random \v{C}ech complex $\Cech_r(\mathcal{P}_n)$ satisfies
$$ E(\beta_k(\mathcal{P}_n,r_n)) \in \Theta \left( n  ({r_n}^{d}n)^{(k+1)} \right ),$$
as $n \rightarrow \infty$.
\end{cor}

The proof of Theorem \ref{thm1} makes essential use of a method used by Penrose to study random geometric graphs \cite{Penrose2003Random}. We generalize this method to the setting of finite geometric properties in Section \ref{sec2}, and we apply it in Section \ref{sec3} to prove the asymptotic lower bound. In Section \ref{sec4} we show the upper bound to complete the proof of Theorem \ref{thm1}, and we conclude our paper with final remarks and open problems in Section \ref{sec5}.

\begin{figure}
\begin{center}
\includegraphics[width=.3\textwidth]{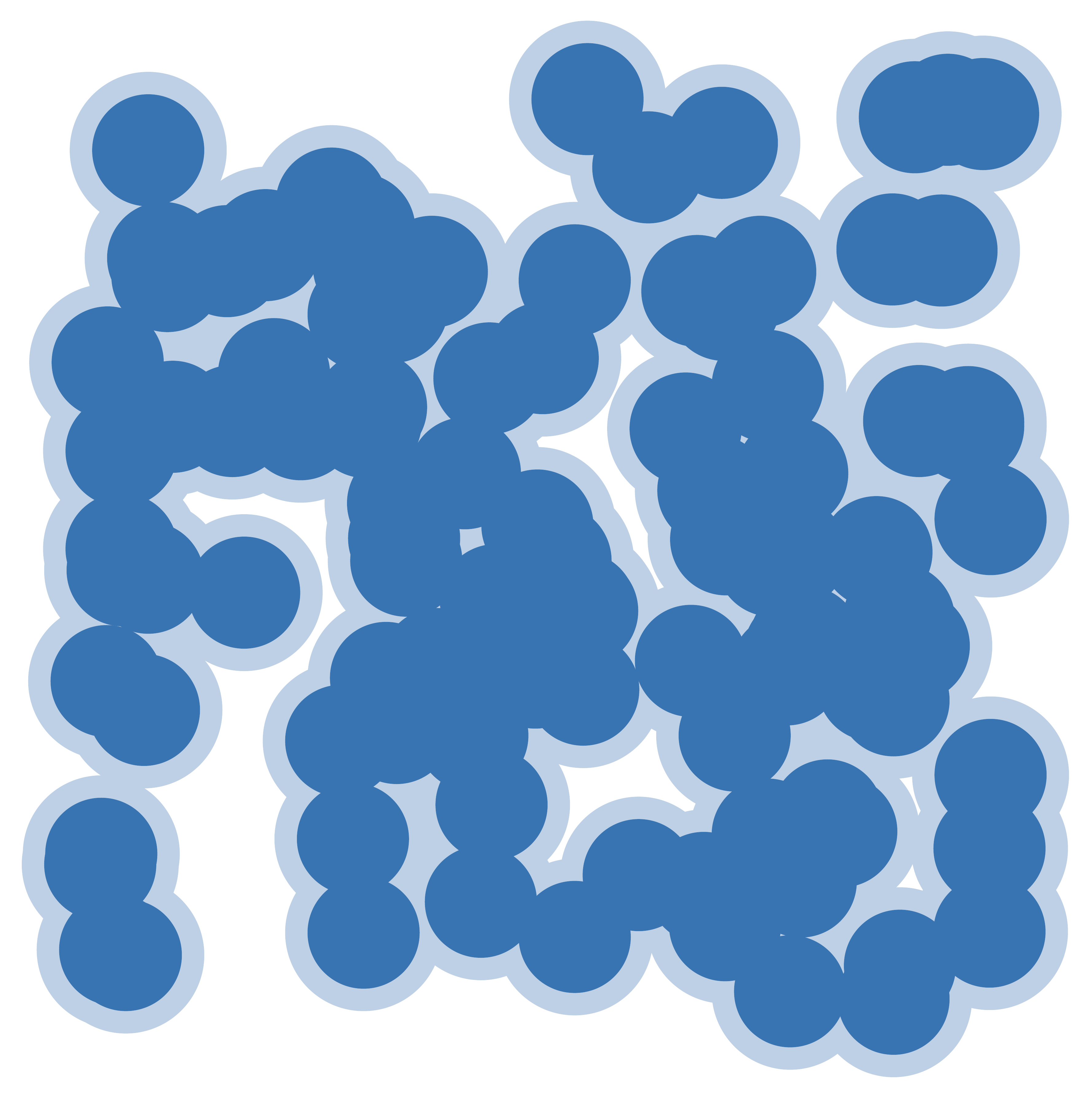}
\hfil
\includegraphics[width=.3\textwidth]{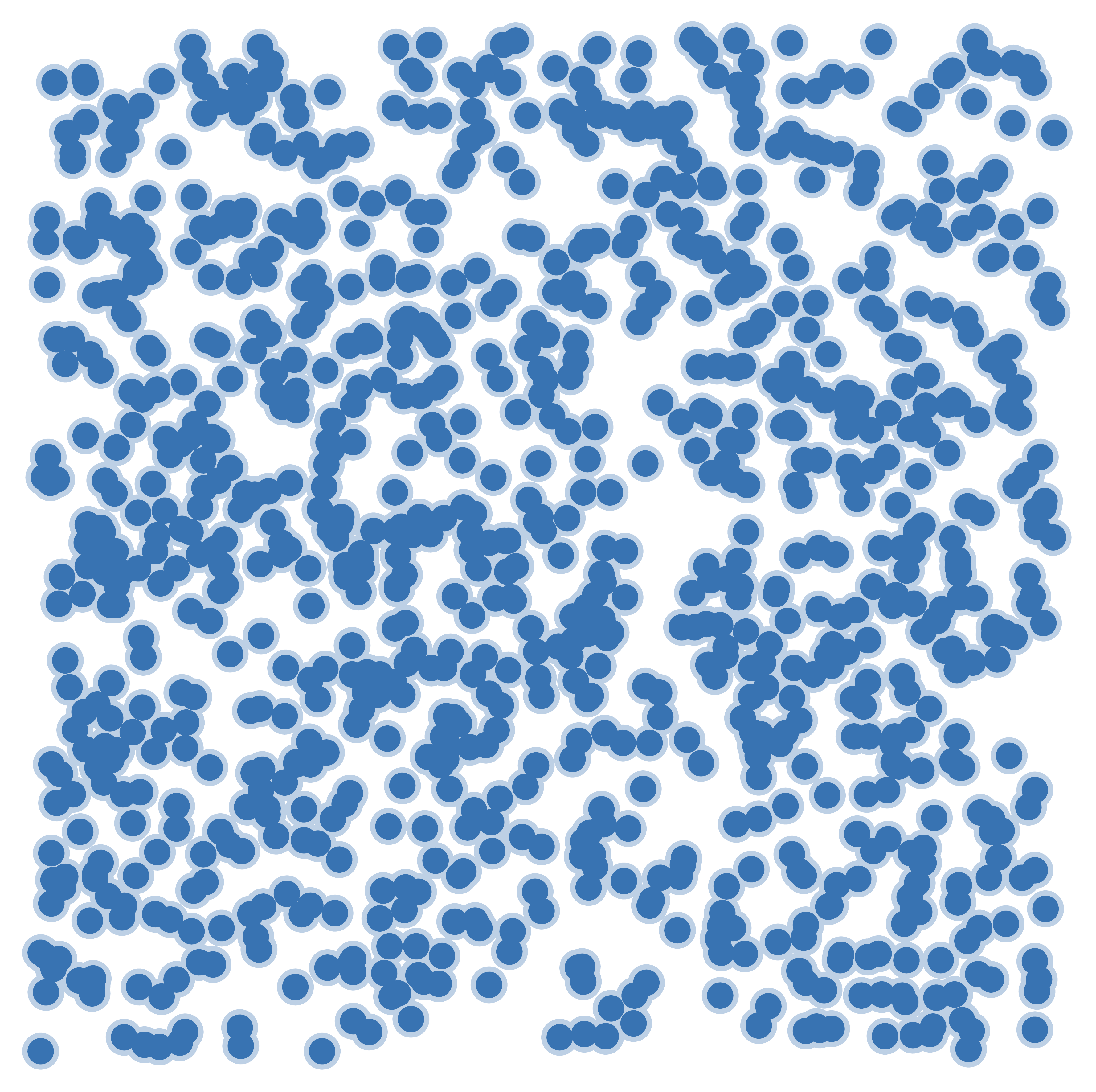}
\hfil
\includegraphics[width=.3\textwidth]{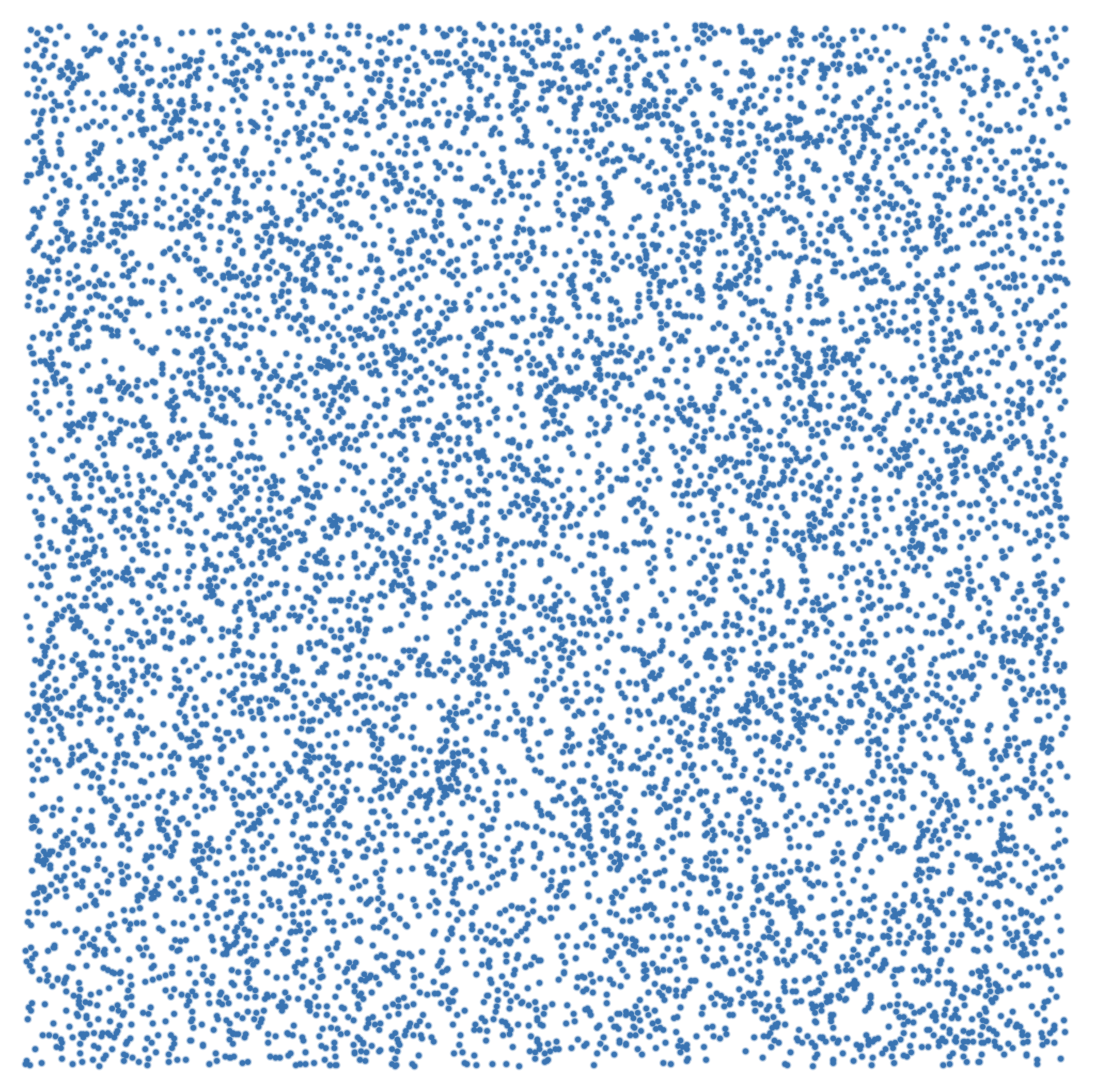}
\end{center}
\caption{Examples illustrating \cref{thm1}. Unions of balls 
of radius $r_n$ (dark) and $\vartheta r_n$ (light) 
around point sets $P_n$ of cardinality $n=100$ (left), $n=1000$ (center), $n=10000$ (right), drawn uniformly at random from $[-1,1]^2$, with parameters chosen such that $E(\beta_k^{\vartheta}(\mathcal{P}_n,r_n)) \in \Theta(1)$; specifically, $d=2, \vartheta = 1.4, k=1, m=m(\vartheta,k)=4,$ and $r_n=c n^q$ with $c=2.6, q=-\frac{m}{d(m-1)}$. In each of the three instances, we have $\beta_k^{\vartheta}(P_n,r_n)=1$.}
\label{fig:example}
\end{figure}

\section{The method of Penrose}
\label{sec2}

The goal of this section is to reformulate two technical results of Penrose \cite[Propositions 3.1 and 3.2]{Penrose2003Random} on the expectation of connected subgraph and component counts in random geometric graphs. In particular, we show that the proof method of Penrose allows us to state these results in a more general way, which we use to prove Theorem \ref{thm1} in the next sections.

\subsection{Finite geometric properties}
Let $\mathrm{P}(\RR^d)$ denote the power set of $\RR^d$.
A family of indicator functions $g_{r,p}: \mathrm{P}(\RR^d) \rightarrow \{ 0,1\}$, with $p \in \NN$ and $0< r \in \RR$, is called a \emph{finite geometric property} if 
\bigskip
\begin{compactenum}[(i)]
\item $g_{r,p}$ is \emph{translation invariant}, i.e., $g_{r,p}(Y)=g_{r,p}(Y+t)$ for all $t \in \RR^d$; 
\item $g_{r,p}$ is \emph{scaling equivariant}, i.e., $g_{r,p}(Y)=g_{\lambda r,p}(\lambda Y)$ for all $\lambda \in \RR$; 
\item $g_{r,p}$ is \emph{finite}, i.e., $|Y| \neq p \Longrightarrow g_{r,p}(Y)=0$ for every $Y \subseteq \RR^d$; in other words $g_{r,p}$ is zero unless $Y$ has $p$ elements;
\item $g_{r,p}$ is \emph{local}, i.e., any $Y \subset \RR^d$ with $g_{r,p}(Y)=1$ satisfies $\diam Y \leq C \cdot r\cdot p$ for some constant $C$.
\end{compactenum}
\bigskip

The first example motivating our definition is the finite geometric property of point configurations inducing a geometric graph isomorphic to a given connected graph $\Gamma$. To this end, let $G(P,r)$ denote 
the \emph{geometric graph on $P$ at scale $r$,} containing an edge precisely for every pair of points with distance at most $r$. 
Define an indicator function for all finite $Y \subset \RR^d$ by 
\begin{equation*}
\mathcal{H}_{r,p}(Y) := \mathbf{1}_{|Y| = p } \cdot \mathbf{1}_{G(Y,r) \cong \Gamma}. 
\end{equation*}
This indicator function checks whether the geometric graph built on the points of the finite set $Y$ is isomorphic to $\Gamma$, and satisfies all required properties.

The results of Penrose \cite{Penrose2003Random} concern the special case of connected geometric graphs. In our framework, we replace the condition that the geometric graph built on a finite set $Y$ is isomorphic to a \emph{connected} graph by the above more general condition of being a \emph{local} property. 
Note that isomorphism to a given connected graph $\Gamma$ is a local property since connectivity enforces the vertex set to have diameter bounded by $r \cdot p$.
The generalization to local properties enables us to study a wider class of indicator functions. For example,
\begin{equation*}
\mathcal{H}'_{r,p}(Y) := \mathbf{1}_{|Y| = p } \cdot \mathbf{1}_{\beta_0(Y,r) = p } \cdot \mathbf{1}_{\diam Y \leq r\cdot p }, 
\end{equation*}
for all finite $Y \subset \RR^d$ is a geometric property, which checks whether $Y$ consists of $p$ points such that all pairwise distances are greater than $r$ but less than or equal to $r\cdot p$.

We now turn to a particular class of finite geometric properties, which are very natural in certain settings. In particular, they appear whenever we are interested in a property that not only concerns the points of a subset $Y \subseteq X$ but also the elements of the complement $X \setminus Y$.

We call a family of indicator functions $h_{r,p}: \mathrm{P}(\RR^d)\times \mathrm{P}(\RR^d) \rightarrow \{ 0,1\}$ for $p \in \NN$ and $0< r \in \RR$ a \emph{finite geometric subset property} if
\bigskip
\begin{compactenum}[(i)]
\item $h_{r,p}(Y,X) = 0$ whenever $Y \not \subseteq X$;
\item $h_{r,p}(Y,X)$ can be written as a product $h_{r,p}(Y,X) = \tilde{h}_{r}(Y,X) \cdot g_{r,p}(Y)$, where $\tilde{h}_{r}$
is an indicator function and $g_{r,p}$ is a finite geometric property;
\item $h_{r,p}$ is \emph{translation invariant}, i.e., $h_{r,p}(Y,X)=h_{r,p}(Y+t,X+t)$ for all $t \in \RR^d$; 
\item $h_{r,p}$ is \emph{scaling equivariant}, i.e., $h_{r,p}(Y,X)=h_{\lambda r,p}(\lambda Y, \lambda X)$ for all $\lambda \in \RR$; 
\item $h_{r,p}$ is \emph{finite}, i.e., $|Y| \neq p \Longrightarrow h_{r,p}(Y,X)=0$ for every $Y \subseteq \RR^d$.
\end{compactenum}
\bigskip
Note that (ii) implies that $h_{r,p}$ is local since $g_{r,p}$ is assumed local as a finite geometric property. 
In contrast to just a geometric property, which depends only on the point set $Y$, a geometric subset property can also depend on the other points in the larger set $X$.

\begin{ex}
As an example which will also be used later, we say a subset $Y \subseteq X$ \emph{forms an isolated component} in the \v Cech complex $\Cech_s(X)$ of a point set~$X$ if the union of balls with radius $s$ centered at the points of $Y$ has an empty intersection with the union of balls centered at points of $X \setminus Y$ and if $\beta_0(Y,s)=1$. 
Then
$$
\comp_{r,p}(Y,X) = \left \{
\begin{array}{cl}
1 & \mbox{if $|Y|=p$ and $Y$ forms an isolated component in $\Cech_r(X)$, } \\
0 & \mbox{otherwise,}
\end{array}
\right.
$$
is a finite geometric subset property, 
which can be written as a product
\[
\comp_{r,p}(Y,X) = \sep_{r}(Y,X) \cdot \conn_{r,p}(Y), \]
where the indicator function $\sep_{r}$ is given by
\[
\sep_{r}(Y,X)=
\begin{cases}
1 & \text{if } d(x,y) > 2r \text{ for all } x \in X\setminus Y \text{ and } y \in Y,\\
0 & \text{otherwise.}
\end{cases}
\]
and
the  finite geometric property $\conn_{r,p}$ is given by 
\[
\conn_{r,p}(Y)=
\begin{cases}
1 & \mbox{if $|Y|=p$ and $\beta_0(Y,r)=1$, } \\
0 & \mbox{otherwise.}
\end{cases}
\]
\end{ex}

\subsection{Subset counts}
We are interested in the expected number of occurrences of a finite geometric property in a random point set, i.e., the number of subsets satisfying the property. Penrose \cite{Penrose2003Random} states our Propositions \ref{prop:penrose1} and \ref{prop:penrose2} in the setting of subgraph counts. His proofs, however, apply to our more general setting with only minor modifications, as we will now show.
For a finite $P \subset \RR^d$ we define
\begin{equation}
\Count(g_{r,p}, P) := \sum_{Y \subseteq P} g_{r,p}(Y) \ \ \ \ \text{ and } \ \ \ \ \SubsetCount(h_{r,p}, P) := \sum_{Y \subseteq P} h_{r,p}(Y,P),
\end{equation}
which count the occurrences of a finite geometric property $g_{r,p}$ and of a finite geometric subset property $h_{r,p}$ in the set $P$, respectively.

Considering the random setting of a Poisson point process $\mathcal{P}_n$ on $\RR^d$ with intensity function $x \mapsto n f(x)$, where $f$ is a bounded probability density function on $\RR^d$, and given a finite geometric property $g_{r,p}$, we define
\begin{equation}
\mu_{g_{1,p}} := \frac{1}{p!} \int_{\RR^d} f(x)^p \, dx \ \int_{(\RR^d)^{p-1}} g_{1,p}( \{ 0, x_1, \ldots, x_{p-1} \} ) \, d(x_1, \ldots, x_{p-1}).
\label{mu}
\end{equation}
Note that the boundedness of $f$ implies that $f(x)^p$ is integrable: Let $f\leq C$ for some constant $C$ and assume by induction that $f^{p-1}$ is integrable. Then $f^p \leq f^{p-1} \cdot C$ is integrable.
Furthermore, the condition that $g_{r,p}$ is local implies that $g_{1,p}( \{ 0, x_1, \ldots, x_{p-1} \})$ is integrable since it is bounded and zero outside of a bounded region, a ball centered at $0$ with radius equal to the diameter bound from the locality condition. We obtain the following asymptotic result:

\begin{prop}[see \cite{Penrose2003Random}, Proposition 3.1] \label{prop:penrose1}
Suppose that $g_{r,p}$ is a finite geometric property for $p\geq 2$ that occurs with positive probability for $\mathcal{X}_p$ and for all sufficiently small $r > 0$. Let $\lim_{n\rightarrow \infty} r_n = 0$. Then
\begin{equation*}
\underset{n \rightarrow \infty} \lim \frac{E(\Count(g_{r_n,p}, \mathcal{X}_n)) }{n ({r_n}^d n)^{(p-1)} } = 
\underset{n \rightarrow \infty} \lim \frac{E(\Count(g_{r_n,p}, \mathcal{P}_n)) }{n ({r_n}^d n)^{(p-1)} } = 
\mu_{g_{1,p}}.
\end{equation*}
\end{prop}

\begin{proof}
We slightly modify the proof of \cite[Proposition 3.1]{Penrose2003Random}. The goal is to highlight that we can replace the subgraph counts of Penrose by our more abstract finite geometric properties. First, we consider binomial point processes $\mathcal{X}_n$ of fixed size $n$, before applying Palm theory to obtain the desired result for Poisson processes $\mathcal{P}_n$.

Since our geometric property is \emph{finite}, we can write
$$ E(\Count(g_{r_n,p}, \mathcal{X}_n)) = {n \choose p} E(g_{r_n,p}(\mathcal{X}_p)), $$
in which $\mathcal{X}_p$ denotes a set of exactly $p$ random points. Hence,
\begin{align}
\notag E(\Count(g_{r_n,p}, \mathcal{X}_n)) &=  {n \choose p} \int_{\RR^d} \ldots \int_{\RR^d} g_{r_n,p}(\{ x_1, \ldots, x_p\})
f(x_1)^p \, dx_p \ldots dx_1 \\
\label{sums} &+ {n \choose p} \int_{\RR^d} \ldots \int_{\RR^d} g_{r_n,p}(\{ x_1, \ldots, x_p\}) \left( \prod_{i=1}^p f(x_i)-f(x_1)^p  \right) \prod_{i=1}^p dx_i.
\end{align}
Applying a change of variables $x_1=x$ and $x_i = x_1 + r_n y_i$ for $2 \leq i \leq p$, the first summand on the right hand side of \eqref{sums} becomes
\begin{align}
\label{firstsummand}
 {n \choose p} {r_n}^{d(p-1)} \int_{\RR^d} \ldots \int_{\RR^d} g_{r_n,p}(\{ x, x+r_n y_2, \ldots, x+r_n y_p\}) \, dy_p \ldots dy_2 \ f(x)^p \, dx.
\end{align}
Since the finite geometric property is {translation invariant} and {scaling equivariant}, we have $g_{r_n,p}(\{ x, x+r_n y_2, \ldots, x+r_n y_p\}) = g_{1,p}(\{ 0, y_2, \ldots, y_p\})$.
Comparing with \eqref{mu}, we can rewrite~\eqref{firstsummand} as
$$
{p!}{{n \choose p}}{r_n}^{d(p-1)}\mu_{g_{1,p}} $$
and conclude that the first summand on the right hand side of \eqref{sums} divided by $n^p {r_n}^{d(p-1)} \mu_{g_{1,p}}$ converges to $1$ as $n$ goes to infinity.

Moreover, the second summand in \eqref{sums} multiplied by ${n}^p {r_n}^{d(p-1)}$ tends to zero, as shown in the proof of \cite[Proposition 3.1]{Penrose2003Random}; we omit the argument here, which carries over verbatim to our context. We conclude that
\begin{equation*}
\underset{n \rightarrow \infty} \lim \frac{E(\Count(g_{r_n,p}, \mathcal{X}_n)) }{{n}^p {r_n}^{d(p-1)}} = \mu_{g_{1,p}}.
\end{equation*}

Finally, we use Palm theory to obtain the result for the Poisson process $\mathcal{P}_n$ from the one for the binomial process $\mathcal{X}_n$. 
Applying Theorem \ref{thm:palm} to the function $h: (Y,X) \mapsto g_{r_n,p}(Y)$,
we obtain
$$ E(\Count(g_{r_n,p}, \mathcal{P}_n)) = E \Bigg( \sum_{Y \subseteq \mathcal{P}_n } g_{r_n,p}(Y) \Bigg ) = \frac{n^p}{p!} E(g_{r_n,p}(\mathcal{X}_p)). $$
On the other hand, we have $E(\Count(g_{r_n,p}, \mathcal{X}_n)) = {n \choose p} E(g_{r_n,p}(\mathcal{X}_p))$. Hence, $$\lim_{n \rightarrow \infty} \frac{E(\Count(g_{r_n,p}, \mathcal{P}_n))}{E(\Count(g_{r_n,p}, \mathcal{X}_n))} 
= 1 , $$ and the result follows.
\end{proof}

Next, we prove an analogous proposition for finite geometric subset properties. 
Our proposition generalizes a similar statement by Penrose about the number of connected components with a given isomorphism type in a random geometric graph \cite[Proposition 3.2]{Penrose2003Random}.

\begin{prop}
\label{prop:penrose2}
Suppose that ${h}_{r,p}$, for $p \geq 2$ and $r>0$, is a finite geometric subset property, given as the product 
$${h}_{r,p}(Y,X) = \tilde h_{r}(Y,X) \cdot g_{r,p}(Y)$$ 
of an indicator function $\tilde h_{r}(Y,X)$ and a finite geometric property $g_{r,p}(Y)$.
Assume $n \geq p$, and let $\mathcal{X}_p = \{x_1, \dots , x_p \} \subseteq \mathcal{X}_n$.
If ${g}_{r,p}$ occurs with positive probability for $\mathcal{X}_p$ and for all sufficiently small $r > 0$, and 
if additionally $g_{{r_n},p}(Y)=1$ implies $\tilde h_{{r_n}}(\mathcal{X}_p,\mathcal{X}_n)=1$ asymptotically almost surely as $n \to \infty$, then
\begin{equation*}
\underset{n \rightarrow \infty} \lim \frac{E(\SubsetCount({h}_{r_n,p}, \mathcal{X}_n)) }{n ({r_n}^d n)^{(p-1)}} =
\underset{n \rightarrow \infty} \lim \frac{E(\SubsetCount({h}_{r_n,p}, \mathcal{P}_n)) }{n ({r_n}^d n)^{(p-1)}} =
 \mu_{g_{1,p}}.
\end{equation*}
\end{prop}

\begin{proof}
Let $\mathfrak{A}$ be the event ${h}_{r_n,p}(\mathcal{X}_p,\mathcal{X}_n) = 1$, $\mathfrak{B}$ the event $g_{r_n,p}(\mathcal{X}_p) = 1$, and $\mathfrak{C}$ the event $\tilde h_{r_n}(\mathcal{X}_p,\mathcal{X}_n) = 1$.
Given $\mathfrak{B}$, the conditional probability of event $\mathfrak{A}$ is the conditional probability of event $\mathfrak{C}$, which tends to $1$ by assumption. Hence
\begin{align*} 
P( \mathfrak{A}) &= P( \mathfrak{A} \mid \mathfrak{B} ) P(\mathfrak{B}) = 
P( \mathfrak{C} \mid \mathfrak{B} ) P(\mathfrak{B}) \rightarrow P(\mathfrak{B}) = E(\Count(g_{r_n,p}, \mathcal{X}_n)),
\end{align*}
and we obtain
\begin{align*}
\SwapAboveDisplaySkip
E(\SubsetCount({h}_{{r_n},p}, \mathcal{X}_n)) &= {n \choose p} P(\mathfrak{A})
\in  (1+ o(1) ) E(\Count(g_{r_n,p}, \mathcal{X}_n)).
\end{align*}
The claim now follows from Proposition \ref{prop:penrose1}.
\end{proof}

\begin{ex}\label{ex:graphComponentCount}
As an example, let $\Gamma$ be a fixed, connected graph on $p$ vertices, $p \geq 2$.
We consider the component count $J_n(\Gamma)$, i.e., the number of components isomorphic to $\Gamma$, as studied by Penrose \cite{Penrose2003Random}. 
Using
$$ {h}_{r_n,p}(Y,\mathcal{X}_n) = \sep_{r_n}(Y,\mathcal{X}_n) \cdot \mathbf{1}_{G(Y,r_n) \cong \Gamma},$$
this function can be written as
$$ J_n(\Gamma) = \SubsetCount({h}_{r_n,p}, \mathcal{X}_n)) = \sum_{Y \subseteq \mathcal{X}_n} {h}_{r_n,p}(Y,\mathcal{X}_n). $$
Note that $G(Y,r_n) \cong \Gamma$ a.a.s.\@ implies $\sep_{r_n}(Y,\mathcal{X}_n)=1$, as the volume of the union of balls of radius $2r_n$ around $X \setminus Y$ goes to zero as $n \to \infty$, and hence the probability that this set contains a point in $Y$.
Thus, Proposition \ref{prop:penrose2} gives the same result as \cite[Proposition 3.2]{Penrose2003Random}, namely
$$ \underset{n \rightarrow \infty} \lim \frac{E( J_n(\Gamma) ) }{n ({r_n}^d n)^{(p-1)} } = \mu_{\mathbf{1}_{G(Y,1) \cong \Gamma}}. $$
\end{ex}

\section{Lower Bound}
\label{sec3}

In this section we obtain an asymptotic lower bound for the expected persistent Betti number $\beta_k^{\vartheta}(\mathcal{P}_n,r_n)$. We refer to  Munkres \cite{Munkres1984Elements} as a reference on simplicial homology. In the following, we use $C_k(K)$, $Z_k(K)$, and $B_k(K)$ to denote the $k$-chains, $k$-cycles, and $k$-boundaries of a simplicial complex $K$ with coefficients in a fixed field $\mathbb K$.

\subsection{Minimal $\vartheta$-persistent cycles}
\begin{defi}
Let $P \subset \mathbb R^d$ be such that there exists an $r>0$, a $\vartheta \geq 1$, and a non-bounding cycle 
$$\gamma \in Z_k(\Cech_r(P) ) \setminus B_k(\Cech_{\vartheta r}(P) ).$$ 
Then we say that the point set $P$ \emph{forms the $\vartheta$-persistent $k$-cycle} $\gamma$ in $\Cech_r(P)$.
\end{defi}
We define
\begin{align*}
m(\vartheta, k) :=  \min \{|P| : P \text{ forms a $\vartheta$-persistent cycle} \};
\end{align*}
that is, $m(\vartheta,k)$ is the minimal number of vertices needed to form a $k$-cycle with persistence greater than $\vartheta$. 
The quantity $m(\vartheta, k)$ plays an important role in
\cite[Sections 4 and 5]{Bobrowski2016Maximally}, where the authors provide (somewhat implicitly) lower and upper bounds on this number.
For our purposes, the following argument providing an upper bound is sufficient:

\begin{lem} 
\label{lem:thetaPersistentCyclesExist}
For every $\vartheta \geq 1$, $m(\vartheta, k)$ is bounded from above. 
\end{lem}

\begin{proof}
Consider the standard $k+1$-simplex $\Delta$ in $\mathbb R^{k+2} \subset \mathbb R^d$. 
Choose any $r>0$ such that $\vartheta r$ is less than the distance %
from the circumcenter of the simplex to any point on its boundary, $\bd \Delta$. 
For example, it suffices to choose $r<\frac{1}{2\vartheta}$.
We may choose $i \in \mathbb N$ such that the simplices of the $i$th barycentric subdivision of $\bd \Delta$ all have diameter at most $r$ (see, e.g., \cite[Chapter 3]{Spanier}); let $Q$ be the vertices of the resulting subdivision. Now $Q \subset \bd \Delta \subset \Ball_r(Q)$ but $\Delta \not\subset \Ball_{\vartheta r}(Q)$ by the choice of $r$, so $\bd \Delta$ supports a $\vartheta$-persistent $k$-cycle formed by $Q$.
\end{proof}

To simplify notation, from now on we assume that $1 \leq k \leq d-1$ and $\vartheta \geq 1$ are fixed, and set $m = m(\vartheta,k)$.

\begin{defi}
Let $P \subset \mathbb R^d$.
A subset $P'=\{ x_1, x_2, \ldots, x_p \} \subset P$ is said to \emph{form a $\vartheta$-persistent isolated $k$-cycle in $\Cech_r(P)$} if it forms a $\vartheta$-persistent $k$-cycle in $\Cech_r(P')$ and $P'$ is a $\vartheta r$-isolated subset of $P$.
If $p=m$, we call this cycle \emph{minimal}.

\end{defi}

The condition that $P'$ is $\vartheta r$-isolated is equivalent to the condition that there are no edges between $P'$ and $P \setminus P'$ in $\Cech_{\vartheta r}(P)$.
Note that this definition therefore ensures that a $\vartheta$-persistent isolated $k$-cycle formed by $P'$ is not only non-bounding in $\Cech_{\vartheta r}(P')$, but also in $\Cech_{\vartheta r}(P)$.

\begin{lem} \label{lem:prob}
Let $\vartheta \geq 1$ and $0 < r < \frac{1}{2\vartheta}$. Then $\mathcal{P}_n$ contains with positive probability a subset of $m=m(\vartheta, k)$ vertices forming a minimal $\vartheta$-persistent isolated  $k$-cycle in $\Cech_r(\mathcal{P}_n)$.
\end{lem}

\begin{proof}
Fixing $\vartheta$, we show that, with positive probability, there exists a minimal $\vartheta$-persistent isolated $k$-cycle.
By the construction in the proof of \cref{lem:thetaPersistentCyclesExist}, there exists a point set $Q=\{ y_1, y_2, \ldots, y_{m} \}$ and a radius $R > \vartheta r$ such that 
\[
\gamma \in Z_k(\Cech_r(Q) ) \setminus B_k(\Cech_{R}(Q) );
\]
equivalently, the induced homomorphism $H_k( \Cech_r(Q) \hookrightarrow \Cech_R(Q) )$ is nonzero, and hence $Q$ forms a $\vartheta$-persistent $k$-cycle in $\Cech_{r}(Q)$.

For each point $y_i$ in $Q$, consider a point $x_i$ contained in an open ball with radius $\delta$ centered at $y_i$, where $\delta$ is chosen such that $\frac{R-\delta}{r+\delta} = \vartheta; $
explicitly,
$ \delta =  \frac{R-\vartheta r}{\vartheta +1}. $
By stability of persistent homology \cite{stab}, perturbing the points by $\delta$ also changes the persistent homology by at most $\delta$.
Specifically, for each such choice $\{ x_1, \ldots, x_{m} \}=P$ we have 
\[\Cech_r(Q) \subseteq \Cech_{r+\delta}(P) \subseteq \Cech_{R-\delta}(P) \subseteq \Cech_R(Q).\]
Since the induced homomorphism $H_k( \Cech_r(Q) \hookrightarrow \Cech_R(Q) )$ is nonzero, so is the induced homomorphism $H_k( \Cech_{r+\delta}(P) \hookrightarrow \Cech_{R-\delta}(P) )$, implying that $P$ forms a $\vartheta$-persistent $k$-cycle in $\Cech_{r+\delta}(P)$.

The probability that $\mathcal{P}_n$ contains such a subset $P$ is positive since a ball of radius $\delta$ around $y_i$ has positive volume, proving the claim.
\end{proof}

\subsection{Proof of Theorem \ref{thm1}: Lower Bound }

We consider the finite geometric property
$$
\zeta_{r,p}^{\vartheta,k}(Y) = \left \{
\begin{array}{cl}
1 & \mbox{if $|Y|=p$ and $Y$ forms a $\vartheta$-persistent $k$-cycle in $\Cech_{r}(Y)$, } \\
0 & \mbox{otherwise,}
\end{array}
\right.
$$
as well as the finite geometric subset property
$$ \Upsilon_{r,p}^{\vartheta,k} (Y,X) = \sep_{\vartheta r}(Y,X) \cdot \zeta_{r,p}^{\vartheta,k}(Y).$$
If this property is satisfied, we say that \emph{$Y$ forms an isolated $\vartheta$-persistent $k$-cycle in $\Cech_{r}(X)$.}
As in \cref{ex:graphComponentCount}, the property $\zeta_{r,p}^{\vartheta,k}(Y)$ a.a.s.\@ implies $\sep_{r_n}(Y,\mathcal{X}_n)=1$.
Writing $\Upsilon = \Upsilon_{r,m}^{\vartheta,k}$, we consider the function
$$ \SubsetCount(\Upsilon,X) = \sum_{Y \subseteq X} \Upsilon(Y,X),  $$
which counts minimal $\vartheta$-persistent isolated $k$-cycles in $X$.

For $X=\mathcal{P}_n$, we obtain a random variable that bounds the $k$th persistent Betti number $\beta_k^{\vartheta}(\mathcal{P}_n,r)$ of $\Cech_{r}(\mathcal{P}_n) \hookrightarrow \Cech_{\vartheta r}(\mathcal{P}_n)$ from below:
\begin{lem} \label{lem:count}
$ \SubsetCount(\Upsilon_{r,m}^{\vartheta,k},\mathcal{P}_n) \leq \beta_k^{\vartheta}(\mathcal{P}_n,r).$
\end{lem}
\begin{proof}
We first show that the subsets $Y \subseteq \mathcal{P}_n$ having property $\Upsilon = \Upsilon_{r,m}^{\vartheta,k}$ must be pairwise disjoint. To see this, consider such a set~$Y$. Then $\Cech_r(Y)$ must be connected, as $Y$ is a minimal vertex set supporting a $\vartheta$-persistent $k$-cycle: assume that $\Cech_r(Y)$ is the disjoint union of $\Cech_r(Y_1)$ and $\Cech_r(Y_2)$. Then a $\vartheta$-persistent $k$-cycle $\gamma \in Z_k(\Cech_r(Y) )$ splits into $\gamma=\gamma_1+\gamma_2$ with $\gamma_i \in Z_k(\Cech_r(Y_i))$, $i=1,2$, at least one of them being nonbounding in $\Cech_{\vartheta r}(Y)$. By minimality of $Y$, we must have that one of the $Y_i$ is $Y$ and the other one is empty.
Now since the sets $Y$ with property $\Upsilon$ also form isolated connected components in $\Cech_{\vartheta r}(\mathcal{P}_n)$, they must be pairwise disjoint.

Consequently, the persistent cycles formed by these subsets generate linearly independent homology classes, which all contribute to $\beta_k^{\vartheta}(\mathcal{P}_n,r)$, and the claim follows.
\end{proof}

The asymptotic behavior of $E( \SubsetCount(\Upsilon,\mathcal{P}_n) )$ in terms of $n$ and $r_n$ follows directly by applying Lemma \ref{lem:prob} and Proposition \ref{prop:penrose2} to the finite geometric subset property $\Upsilon_{r,m}^{\vartheta,k}$:  
for $k \geq 1$, we have
\begin{equation*}
\lim_{n \rightarrow \infty} \frac{ E( \SubsetCount(\Upsilon,\mathcal{P}_n) )} { n^m  {r_n}^{d(m-1) } } = C,
\end{equation*}
where $C$ is a constant depending only on $k$ and $d$. 
Thus, we have 
\[E( \SubsetCount(\Upsilon,\mathcal{P}_n) ) \in \Theta( n^m  {r_n}^{d(m-1) }  ),\] from which the lower bound on $\beta_k^{\vartheta}(\mathcal{P}_n,r_n)$ follows with Lemma \ref{lem:count}:
\begin{cor}
\label{cor:lowerBound}
$ E(\beta_k^{\vartheta}(\mathcal{P}_n,r_n)) \in \Omega( n^m  {r_n}^{d(m-1) }  ). $
\end{cor}

\section{Upper Bound}
\label{sec4}

Finally, in this section we derive an asymptotic upper bound for $E(\beta_k^{\vartheta}(\mathcal{P}_n,r_n))$ by another application of the methods introduced in Section \ref{sec2}. 

\begin{lem}
\label{lem:upperBound}
$ E(\beta_k^{\vartheta}(\mathcal{P}_n,r_n)) \in \mathcal O( n^m  {r_n}^{d(m-1) }  ).
$
\end{lem}

\begin{proof}
Let $\vartheta \geq 1$, $d \geq 2$ and $1 \leq k \leq d-1$.

Any contribution to $\beta_k^{\vartheta}(\mathcal{P}_n,r_n)$ necessarily comes from a $k$-cycle supported on a set of $m$ or more vertices. 
Let $c_{q,k}^\vartheta$ be the maximal $\vartheta$-persistent Betti number $\beta_k^{\vartheta}(P_q,r)$ of any point set $P_q$ of $q$ vertices and any $r \geq 0$.
Note that 
\[\beta_k^{\vartheta}(P_q,r) \leq c_{q,k}^\vartheta \leq \dim C_k(\Cech_r(P_q)) \leq {q \choose k+1 }.\] 
Since any $\vartheta$-persistent $k$-cycle requires at least $m$ vertices, a point set $P_p$ of $p \geq m$ vertices can have Betti number at most
\begin{equation} \label{constant}
\beta_k^{\vartheta}(P_p,r)\leq C^\vartheta_{p,k} := \sum_{i=0}^{p-m} {p \choose m+i } c_{m+i,k}^\vartheta
\end{equation}
$k$-cycles supported on $m$ or more vertices.

Using the geometric subset property $\Upsilon_{r,p}^{\vartheta,k}$ defined before, we obtain a bound on the persistent Betti number for finite points sets $P$ with arbitrary cardinality:

\[\beta_k^{\vartheta}(P,r) \leq \sum_{p \geq m} \SubsetCount(\Upsilon_{r,p}^{\vartheta,k},P)  \cdot C^\vartheta_{p,k}.\]

Using Penrose's method, we can estimate the asymptotic order of the number of subsets of $\mathcal{P}_n$ with nontrivial $\vartheta$-persistent homology for ${r_n}^d n \rightarrow 0$ as $n \rightarrow \infty$.
Applying \cref{prop:penrose2}, the summand for $p=m$ is dominant, and  we obtain

$$E\left(\sum_{p \geq m} \SubsetCount(\Upsilon_{{r_n},p}^{\vartheta,k},\mathcal{P}_n) \right)
 \in \mathcal{O}\left( n^m {r_n}^{d(m-1)} \right).$$ 
The claim now follows.
\end{proof}

\begin{proof}[Proof of \cref{thm1}]
Combining \cref{cor:lowerBound,lem:upperBound}, we obtain
\[ E(\beta_k^{\vartheta}(\mathcal{P}_n,r_n)) \in \Theta( n^m  {r_n}^{d(m-1) }  )
\]
as claimed.
\end{proof}

\section{Concluding remarks}
\label{sec5}

We conclude with three final remarks. Note that we only determine the asymptotic order of the $k$th $\vartheta$-persistent Betti number of a random \v Cech complex. Getting the same leading terms in the upper and lower bound requires an answer to the following question:
\emph{How many linearly independent $\vartheta$-persistent $k$ cycles can be supported on a set of $m(\vartheta,k)$ vertices?}
We conjecture that the answer is 1, i.e., $c_{m,k}^{\vartheta}=1$. 

A second interesting problem for future investigation is the study of persistent Betti numbers in the \emph{critical} case; i.e., for ${r_n}^d n \rightarrow c$, where $c >0$ is a constant.

Finally, we remark that our arguments can be adjusted to apply to homology with other coefficients, in particular, integer coefficients. An analogous result requires to adjust the definitions accordingly, in particular the definitions of $\vartheta$-persistent cycles, persistent Betti numbers, and the number $m(\vartheta,k)$. 

\subsection*{Acknowledgements}
This research has been supported by the DFG Collaborative Research Center SFB/TRR 109 ``Discretization in Geometry and Dynamics''.
We thank Matthew Kahle for helpful discussions on the upper bound for the persistent Betti number. 

\bibliographystyle{abbrv}

\bibliography{PersBetti}

\todos

\end{document}